\documentclass[12pt]{amsart}

\usepackage{amscd,amssymb,amsfonts,amsmath,amstext,amsthm}

\newtheorem{theorem}{Theorem}[section]
\newtheorem{proposition}[theorem]{Proposition}
\newtheorem{lemma}[theorem]{Lemma}
\newtheorem{corollary}[theorem]{Corollary}

\newtheorem{example}[theorem]{Example}

\newtheorem{remark}[theorem]{Remark}

\renewenvironment{proof}{{\noindent\bf Proof.}}{\hfill $\Box$\par\vskip3mm}

\newcommand{\ot}{\otimes}

\newcommand{\eps}{\varepsilon}

\newcommand{\La}{\Lambda}

\newcommand{\la}{\lambda}

\def\RR{{\mathbb R}}

\def\CC{{\mathbb C}}
\def\ZZ{{\mathbb Z}}
\def\QQ{{\mathbb Q}}

\newcommand{\num}{\refstepcounter{theorem}}

\begin{document}

\title{$FSZ$-groups and Frobenius-Schur indicators of Quantum Doubles}

\begin{abstract}
We study the higher Frobenius-Schur indicators of the representations of the 
Drinfel'd double of a finite group $G$, in particular the question as to when 
all the indicators are integers.\ This turns out to be an interesting group-theoretic question.\
We show that many groups have 
this property, such as alternating and symmetric groups, $PSL_2(q)$, $M_{11}$, $M_{12}$
and regular nilpotent groups. However we show there is an irregular nilpotent group of 
order $5^6$ with non-integer indicators.
\end{abstract}

\author{M. Iovanov}
\address{University of Bucharest, Romania and Department of Mathematics, University of Iowa, Iowa City, IA}
\email{yovanov@gmail.com}

\author{G. Mason}
\address{Department of Mathematics, University of California at Santa Cruz, CA }
\email{gem@ucsc.edu}

\author{S. Montgomery}
\address{Department of Mathematics, University of Southern California, Los Angeles, CA }
\email{smontgom@usc.edu}
\thanks{The first author was supported by an AMS-Simons travel grant and by CNCS grant TE45, no.88/02.02.2010.\ The second and third authors were supported by the NSF}

\thanks{2000 \textit{MSC
Primary 20D99, 20C99; Secondary 16T05}}
\date{}
\maketitle



\section{Introduction}\label{s0}

In the representation theory of finite groups, Frobenius-Schur indicators and their higher analogs are an important set of invariants. For Hopf algebras, the second indicator was generalized in \cite{LM}, and higher indicators in \cite{KSZ}.
They were then extended to quasi-Hopf algebras \cite{MN05} and tensor categories \cite{NS07,NS11}. An important property is that they are gauge invariants, that is, invariants of the tensor category of representations. They proved to have many interesting applications; for example, in \cite{KSZ} it is proved that for $H$ a semi-simple Hopf algebra over $\mathbb{C}$, the exponent of $H$ and the dimension of $H$ have the same prime factors, answering a question of \cite{EG}. They can also be used to distinguish between the tensor categories having the same Grothendieck ring, such as the 8 types of non-isomorphic semi-simple quasi-Hopf algebras (tensor categories) of dimension 8 over 
$\mathbb{C}$. Thus knowing the possible values of the indicators is an interesting problem. In this paper 
we study the case of $H = D(G)$, the Drinfel'd double of a finite group $G$. 

Throughout we work over the complex numbers $\CC$. 
Let $H$ be a semisimple Hopf algebra over $\CC$ and let $V$ be an irreducible representation of $H$ with 
character $\chi$. As defined in \cite{KSZ}, the $n^{th}$ Frobenius-Schur indicator of $V$ is given by
$$\nu_n(V) : =  \chi(\La^{[n]}), $$
where $n$ is a positive integer, $\La$ is a normalized integral of $H$ and for any $h \in H$, $h^{[m]} =h_1h_2\dots h_m.$

\medskip
In the classical case of a group algebra $H =\CC G$ for a finite group $G$,
 $\nu_n(V) : =  \chi(|G|^{-1} \sum_{g\in G} g^n)$, and it is known (cf.\ (\ref{zeta'form}) below) that all higher indicators are integers. 
 
 \medskip
For arbitrary Hopf algebras this is not the case.\ Let $e = e(H)$ denote the exponent of $H$: this is the smallest number $m$ such that 
$h^{[m]}=\eps(h) 1$ for all $h\in H$, and always exists if $H$ is semisimple.\  \cite[Section 3.2]{KSZ} show that $\nu_n(V)$ is an algebraic integer in $\QQ(\zeta_e)$, the $e^{th}$ cyclotomic field, and moreover $\nu_n(V)\in \QQ(\zeta_{e/\gcd(n,e)}) \cap \QQ(\zeta_n)$.\ They construct an example in which the value is not a rational integer. 

\medskip
In this paper we are concerned with the Hopf algebra $D(G)$, the Drinfel'd double of $G$.\ 
It was an open question whether or not all indicator values for $D(G)$ are integers.\ For a positive
integer $n$, we say that $G$
is an $FSZ_n$-group if every $n^{th}$ indicator $\nu_n(V)$ for $D(G)$ is an integer; we say that
$G$ is an $FSZ$-group if it is $FSZ_n$ for all $n$.\ So the question amounts to this: is every
finite group $FSZ_n$ for all $n$?\ We will show that a variety of finite groups indeed have this property.\
These include symmetric and alternating groups, $2$-dimensional linear groups $PSL_2(q)$, the
sporadic simple groups $M_{11}$ and $M_{12}$,
\emph{regular CN}-groups (the centralizer of each nonidentity element is a direct product
of regular $p$-groups in the sense of P.\ Hall \cite{H}), and groups whose
exponent is `not too composite'.\  In a slightly different direction, we show
 that \emph{every} $G$ is $FSZ_n$ for $n=2, 3, 4$ or $6$ (in fact \cite{LM} show $\nu_2(V)\in \{0, \pm1\})$), and that if $G$ is not $FSZ_n$ then there is an
 element $g\in G$ of order $5$ or greater than $7$ such that $C_G(g)$ is not $FSZ_n$.\ These results
 suggest that the class of $FSZ$-groups is extensive.\ Indeed, we have made no attempt
 to be as general as possible, and there are certainly other classes of
 finite groups which can be shown to be $FSZ$ using the methods that we develop in the present paper.\ The most tantalizing open question is whether every finite simple group is $FSZ$.
 
 \medskip
 Despite the proliferation of $FSZ$-groups, it turns out that there are indeed non$FSZ$-groups, indeed
 there are infinitely many of them.\ Because of the results mentioned in the previous paragraph,
 a natural first place to look
for a nonFSZ group would be an irregular $5$-group of order $5^6$ and exponent $5^2$, and we shall see
that indeed there are non$FSZ$ groups of this type.

\medskip
The paper is arranged as follows.\ In Section 2 we give the basic definitions and  some
results about the indicators of $D(G)$ of a combinatorial and character-theoretic nature.\ In Section 3
we use these results to give (Theorem \ref{thmequivs})
various characterizations of $FSZ$-groups.\ We also introduce $FSZ^+$-groups; these are finite groups satisfying a natural group-theoretic condition that is qualitatively stronger than the $FSZ$ condition.\ 
$FSZ^+$-groups have a very convenient inductive property (Theorem \ref{thmcent1}) that we use
extensively in Sections 4 and 5 to show that various classes of groups are $FSZ^+$.\
Indeed, we know of no group that is $FSZ$ but \emph{not} $FSZ^+$.\ In Section 6 we show that all symmetric and alternating groups are $FSZ$.\ Section 7 is devoted to counterexamples, i.e.,
the existence of non$FSZ$-groups, including dedicated computer programs that can run on GAP
to test the $FSZ$ property.

\medskip
We thank Siu-Hung Ng and Pavel Etingof for communicating some of
their unpublished results to us.\  Etingof's result, that wreathed products 
$S_n\wr A$ ($A$ abelian) are $FSZ$, is discussed in Section 6.

\medskip

\section{The sets $G_n(u, g)$}\label{s1} 
From now on, $G$ will always denote a finite group.\  Let $I(G)$ denote the \emph{irreducible characters} of $G, R(G)$ the ring of 
\emph{virtual characters} of $G$ ($\ZZ$-linear combinations of irreducible characters), $cf(G)$ the $\mathbb{C}$-algebra of \emph{class functions} of $G$, and
$\langle \ , \ \rangle_G$ the usual inner product defined on $cf(G)$.

\medskip
Fix a conjugacy class $\mathcal{C}$ of $G$ and an element $u \in \mathcal{C}$.\ Let $C = C(u)$ be the centralizer of 
$u$ in $G$, and let $W$ be an irreducible representation of $C(u)$.\ Then with a suitable action of 
$\CC^G$, $V = \CC G \ot_{\CC C(u)} W$ is an irreducible representation of $D(G)$; moreover all irreducible representations of $D(G)$ are 
obtained in this way \cite{KMM}.

\medskip
Let $\eta$ denote the character of $C$ furnished by $W$, and $\chi= \chi_{\eta}$ denote the corresponding $D(G)$-character of $V$.\ Given this data, we use the notation $V=V_{C, \eta}$ and $W = W_\eta$.
\medskip
For $g \in G$, define (as in \cite[7.2]{KSZ})
\begin{eqnarray*}
G_n(u, g)&=&\{a \in G \ | \ (au^{-1})^n= a^n = g\},
 \end{eqnarray*}
with
\begin{equation*}G_n(u)= \bigcup_{g\in G} G_n(u,g) =\{a \in G \ | \ (au^{-1})^n=a^n \}.
\end{equation*}
For a subgroup $H\subseteq G$ we also set
\begin{eqnarray*}
G^H_n(u, g)&=&\{a \in H \ | \ (au^{-1})^n= a^n = g\}.
 \end{eqnarray*}

It is easy to see that if $a \in G_n(u)$, then $a^n \in C$. Thus $G_n(u, g)=\phi$ unless $g\in C$, 
$G_n(u, g)= G_n^{C(g)}(u, g)$, and
$$G_n(u) = \bigcup_{g\in C} \ G_n(u,g).$$
 
 \medskip
Combining the formula for $\nu_n(V)$ given in \cite[7.4]{KSZ} with our remarks about $G_n(u)$, we see
 \begin{eqnarray} 
\num \label{FSdef2}\nu_n(\chi_{\eta})&=&|C|^{-1}\sum\limits_{g\in C}|G_n(u,g)|\eta(g)\\
\num  \label{FSdef3}
&=& |C|^{-1}\sum_{a\in G_n(u)} \eta(a^n). 
\end{eqnarray}

\medskip
The following character-theoretic characterization of the $n^{th}$ indicator 
will be useful.\ First we introduce
\begin{eqnarray*}\label{zeta}
\zeta_n^C= \zeta_n: C\rightarrow \mathbb{Z},\  g \mapsto |G_n(u, g)|.
\end{eqnarray*}
Note that $\zeta_n \in cf(C)$.
\begin{proposition}\label{char} The following hold.
\begin{eqnarray} 
\num \label{zetaform} \zeta_n &=& \sum_{\eta \in I(C)} \nu_n(\chi_{\eta})\eta,\\
\num  \label{FSdef4} \nu_n(\chi_{\eta}) &=& \langle \zeta_n, \eta \rangle_C.
\end{eqnarray}
\end{proposition}
\begin{proof}\
Because $\zeta_n$ is a class function, there is a decomposition 
$\zeta_n = \sum_{\eta \in I(C)} \alpha_n(\eta)\eta,$
 for scalars $\alpha_n(\eta)$.\ Since $\alpha_n(\eta) = \langle \zeta_n, \eta \rangle_C = |
C|^{-1}\sum_{g\in C} |G_n(u, g)|\eta(g)$, we must  have $\alpha_n (\eta)= \nu_n(\chi_{\eta})$ by  
(\ref{FSdef2}),
 whence (\ref{zetaform}) and (\ref{FSdef4})  follow. 
\end{proof}

\medskip
It will be worthwhile to make explicit the special case  when $u=1$,
corresponding to the original
$n^{th}$  indicators $\nu^G(\eta)$ of the irreducible characters $\eta$ of  $G$.
(Here, $\eta$ is also an irreducible character for $D(G)$ and coincides with
$\chi_{\eta}$). In this case we write
$\varphi_n^G$ in place of $\zeta_n$.\ Then
\begin{eqnarray}
\num \label{zeta'form} \varphi^G_n = \sum_{\eta\in I(G)} \nu^G_n(\eta)\eta,
\end{eqnarray}
and it is well-known that $\varphi^G_n\in R(G)$.\ Indeed by (\ref{FSdef3}),
$\nu^G_n(\eta)  = \langle 1, \lambda^n(\eta) \rangle_G$,
where $1$ is the trivial $1$-dimensional character and $\lambda^n(\eta)$ is the virtual character given by the $n^{th}$ Adams operator applied to $\eta$.\ (Cf.\ \cite[Section 12.B]{CR} for Adams operators.)

\begin{lemma}\label{z's}  There is a bijection of sets
\begin{eqnarray*}
G_n(u, g) \rightarrow G_n(u, g^{-1}),\ \ a \mapsto ua^{-1}.
\end{eqnarray*}
\end{lemma}
\begin{proof}\ This is straightforward to check.\ The inverse  map is $b\mapsto b^{-1}u$.
\end{proof}

\begin{remark}
 \label{p.real}
The indicators of all representations of $D(G)$ are real. 
\end{remark}
\noindent
Indeed, apply complex conjugation to (\ref{FSdef2}).\ Since $\overline{\eta(y)} =  \eta(y^{-1})$, we see 
$\overline{\nu_n(\chi_{\eta})}= \nu_n(\chi_{\eta})\in\RR$.
(S-H.\ Ng has shown that the indicators of the double $D(H)$ for any $H$ are real \cite{N}; our special case is much easier to prove.) 

\bigskip
We also give a simpler proof  for $D(G)$ of the result of \cite{KSZ} that $\nu_n(V)\in\QQ(\zeta_{e/\gcd(n,e)})$. It will follow that to determine whether $G$ is
an $FSZ$-group, it suffices to consider $\nu_d(V)$ for $d \ | \ e$.

\begin{lemma} \label{gal}
Let $e$ be the exponent of $G$, $n$ a positive integer and $d=\gcd(n,e)$.  For any 
irreducible representation $V$ of $D(G)$,  $\nu_d(V),$ $\nu_n(V)\in \QQ(\zeta_{e/d})$ and  there is $\sigma\in{\rm Gal}(\QQ(\zeta_{e/d})/\QQ)$ depending only on $n$ such that $\nu_n(V)=\sigma(\nu_d(V))$. 

 In particular, to show all indicators of $V$ are integers it is enough to show that 
 $\nu_d(V)\in\ZZ$ for all $d \ | \  e$.
\end{lemma}

\begin{proof}
We use the formula in equation (\ref{FSdef3})  for the indicators of a general irreducible representation of $D(G)$. Let $d=\alpha n+\beta e$; we note that $h^n=(hu^{-1})^n$ implies $h^d=h^{n \alpha+e\beta} = (h^n)^\alpha =(hu^{-1})^{n\alpha}= (hu^{-1})^{n\alpha+e\beta}=(hu^{-1})^d$. Hence $G_n(u)=G_d(u)$. So $\sum\limits_{h\in G_n(u)}\eta(h^n)=\sum\limits_{h\in G_d(u)}\eta((h^d)^q)$, where $n=dq$; moreover, $h^d=(h^n)^\alpha\in C(u)$. Since $(h^d)^{e/d}=1$, we see that $\eta(h^d)\in \QQ(\zeta_{e/d})$, since it is 
the trace of a matrix of order $e/d$. Since $\gcd(e/d,q)=1$, $\sigma=\{ t\mapsto t^q\} \in{\rm Gal}(\QQ(\zeta_{e/d})/\QQ)$, and we obviously have $\sigma(\eta(h^d))=\eta(h^n)$. Therefore, $\sigma(\nu_d(V))=\nu_n(V)$. 
\end{proof}

This fact was used in \cite{C} and in some of our computations in Section 7.
\bigskip

\section{Characterizations of $FSZ$-groups}\label{s2}

In this Section we give various characterizations of $FSZ$-groups.\ We also introduce the idea 
of $FSZ^+$-groups; these are special kinds of $FSZ$-groups that are very
natural and are better behaved from a group-theoretic perspective (cf.\ Theorem \ref{thmcent1}).\
It will turn out that for most of the examples considered in the following Sections, it is the stronger
$FSZ^+$ property that we will establish.\ Indeed, we do not know of a group that is
$FSZ$ but \emph{not} $FSZ^+$.

\medskip
We begin with some character-theoretic characterizations.\ We fix some notation which is standard.\ Let $H$ be a subgroup of $C=C(u)$. If $\eta$ is a character of $C$, then $Res^C_H(\eta)$ is the 
restriction of $\eta$ to $H$; conversely for $\la$ a character of $H$, $Ind^C_H(\la)$ denotes the induced character 
from $H$ to $C$. 

\begin{theorem}\label{thmequivs} Let $n$ be a positive integer.\ The following are equivalent.
\begin{eqnarray*}
&&(1)\ \  \nu_n(\chi_{\eta}) \in \mathbb{Z}\ \mbox{for each}\ \eta\in I(C), \\
&&(2)\ \ \zeta_n \in R(C), \\
&&(3)\ \ \langle Res^C_{H}\zeta_n, \la \rangle \in \mathbb{Q}\ \mbox{for each cyclic}\ H\subseteq C\ \mbox{and}\ \la\in I(H), \\
&&(4)\ \ \zeta_n\ \mbox{is constant on rational conjugacy classes of $C$}.
\end{eqnarray*}
\end{theorem}

\begin{proof}\ The equivalence of (1) and (2) is an immediate consequence of
(\ref{zetaform}).\ Moreover, if $\zeta_n \in R(C)$ then 
$Res^C_H\zeta_n \in R(H)$ for any subgroup $H\subseteq C$, and therefore
$\langle Res^C_{H}\zeta_n, \la \rangle \in \mathbb{Z}$ for $\lambda\in I(H)$.\ In particular,
(2)$\Rightarrow$(3).

\medskip
Let $\eta\in I(C)$.\ By Artin's theorem, there are cyclic subgroups $H_{i, \eta}\subseteq C$, 
$\la_{i, \eta} \in I(H_{i, \eta})$ and $c_{i, \eta}\in \mathbb{Q}$ such that $\eta= \sum_i c_{i, \eta} Ind^C_{H_{i, \eta}}
(\la_{i, \eta})$.\ Using (\ref{FSdef3}),  we have
\begin{eqnarray*}
\nu_n(\chi_{\eta}) =\langle \zeta_n, \eta\rangle_C = \sum_i c_{i, \eta} \langle \zeta_n, Ind_{H_{i, \eta}}^C (\la_{i, \eta}) \rangle_C 
= \sum_i c_{i, \eta}\langle Res^C_{H_{i, \eta}}(\zeta_n), \la_{i, \eta} \rangle_{H_{i, \eta}},
\end{eqnarray*}
where the third equality uses Frobenius Reciprocity \cite[7.2]{S}.\
Since $\nu_n(\chi_{\eta})$ is an algebraic integer, it follows that a sufficient condition for
$\nu_n(\chi_{\eta})\in \mathbb{Z}$ is that for each $i$ we have 
$\langle Res^C_{H_{i, \eta}}(\zeta_n), \la_{i, \eta}\rangle_{H_{i, \eta}}\in \mathbb{Q}$.\ Hence,
 (3)$\Rightarrow$(1).

\medskip
We have seen that $\zeta_n$ is a rational-valued class function on $C$.\
Then it is well-known
(\cite{S}, Proposition 34 and Theorem 29) that (4) holds if, and only if,
$\zeta_n$ is a $\mathbb{Q}$-linear combination of 
(rational) characters of $C$.\ This in turn is equivalent to the statement
$\nu_n(\chi_{\eta})\in \mathbb{Q}$ for $\eta\in I(C)$ thanks to (\ref{zetaform}), and because 
each $\nu_n(\chi_{\eta})$ is an algebraic integer,  this is equivalent to (1).\ The Theorem is proved.
\end{proof} 

The following consequence is fundamental for later results.
\begin{corollary} \label{corequiv} Let $n$ be a positive integer.\
Then $G$ is an $FSZ_n$-group if, and only if, for 
all commuting pairs of elements $u, g$ and all integers $m$ with $gcd(m, |G|)=1$, we have
\begin{eqnarray}\num\label{gpeqns} |G_n(u, g)| = |G_n(u, g^m)|.
\end{eqnarray}
\end{corollary}
\begin{proof} This is just an explicit version of Theorem \ref{thmequivs}(4).
\end{proof}

We give two immediate applications of (\ref{gpeqns}).

\begin{lemma}\label{power} Let $n$ be a positive integer, and suppose that
 $G_n(u)$ is closed 
with respect to all power maps $a\mapsto a^m,$ for $gcd(m, |G|)=1$.\ Then (1)-(4)
of Theorem \ref{thmequivs} hold.
\end{lemma}
\begin{proof} Using (\ref{FSdef3}), 
the hypotheses of the Lemma imply that $\nu_n(\eta)$ is invariant
under all Galois automorphisms of the cyclotomic field of order $|G|$ for each
$\eta\in I(C)$.\ It is therefore
rational, and hence integral, as required.\ Alternatively, the assumptions of the Theorem
state that
$a^n=(au^{-1})^n$ implies $a^{mn}=(a^mu^{-1})^n$ whenever $(m, |G|)=1$.
Then (\ref{gpeqns}) holds because the map $x\mapsto x^m$ is bijective on $G$ whenever 
$gcd(m, |G|)=1.$

\end{proof}

\begin{lemma}\label{lemmacent} Let $n$ be a positive integer, and suppose that the centralizer of every element of $G$ of order \emph{not} $1, 2, 3, 4$ or $6$ 
is $FSZ_n$.\ Then $G$ is $FSZ_n$.
\end{lemma}
\begin{proof} We show that (\ref{gpeqns}) holds for every pair of commuting elements
$u, g\in G$.\ Suppose first that  $g$ has order $1, 2, 3, 4$ or $6$.\
Then $g^m=g^{\pm 1}$, and (\ref{gpeqns}) follows from Lemma
\ref{z's}.\ Now suppose that the order of $g$ is $5$ or at least $7$. If $gcd(m, |G|)=1$ then 
$u\in C(g)$ and $G_n(u, g), G_n(u, g^m)\subseteq C(g)$.\ Therefore, $|G_n(u, g)|=|G_n(u, g^m)|$
by Corollary \ref{corequiv} because $C(g)$ is an $FSZ_n$-group
by hypothesis.\ This proves the Lemma.
\end{proof}

To facilitate further applications of Corollary \ref{corequiv}, we introduce the idea of
$FSZ^+$-groups.\ To this end, notice that the subgroup $C_G(u, g)\subseteq G$ that commutes with both
$u$ and $g$ acts by conjugation on $G_n(u, g^m)$.\ We can thus consider two properties
that qualitatively strengthen (\ref{gpeqns}): 
\begin{eqnarray}
 &&\ \ \ \  \mbox{$G_n(u, g)$ and $G_n(u, g^m)$ are \emph{isomorphic} $C_G(u, g)$-sets}.\num \label{Gniso}\\
&&\ \ \ \ \mbox{$G_n(u, g)$ and $G_n(u, g^m)$ afford \emph{isomorphic permutation modules}
for $C_G(u, g)$.} \num \label{Gnpermiso}
\end{eqnarray}

 \noindent
{\bf Remark}.\ It is well-known that (\ref{Gnpermiso}) is a weaker condition than (\ref{Gniso}).\
That is, for a group $H$, isomorphic $H$-sets afford isomorphic permutation modules
for $H$; but the converse is generally false.

\medskip
We call $G$ an $FSZ^+_n$-group if (\ref{Gnpermiso}) holds for all commuting pairs
of elements $u, g\in G$ and all integers $m$ with $gcd(m, |G|)=1$.\ Thus an $FSZ^+_n$-group
is necessarily an $FSZ_n$-group.\ We say that $G$ is an $FSZ^+$-group if it is $FSZ^+_n$ for all $n$.
Such a group is necessarily an $FSZ$-group.\

\begin{theorem}\label{thmcent1} Let $n$ be a positive integer.\ The following are equivalent:
\begin{eqnarray*}
&&(a)\ \mbox{$G$ is an $FSZ^+_n$-group}, \\
&&(b)\ \mbox{The centralizer of every nonidentity element in $G$ is an $FSZ_n$-group}.
\end{eqnarray*}
\end{theorem}
\begin{proof} Condition (\ref{Gnpermiso}) means that each $h \in C(u, g)$ leaves invariant
(i.e., commutes with) the \emph{same number} of elements in both $G_n(u, g)$ and $G_n(u, g^m)$.
Thus, we have
\begin{eqnarray*}
 G \text{ is } FSZ^+_n \ 
 \Leftrightarrow&&  |G_n(u, g) \cap C(h)| = |G_n(u, g^m)\cap C(h)|\ (h \in C(u, g), ug=gu) \\
  \Leftrightarrow&&  |G^{C(h)}_n(u, g)| = |G^{C(h)}_n(u, g^m)|\ (\langle h, u, g\rangle\ \mbox{abelian}) \\
    \Leftrightarrow&&  |G^{C(h)}_n(u, g)| = |G^{C(h)}_n(u, g^m)|\ (u, g \in C(h), ug=gu) \\
    \Leftrightarrow&& C(h)\ \mbox{is $FSZ_n$}\ (h\in G) \\
      \Leftrightarrow&& C(h)\ \mbox{is $FSZ_n$}\ (1\not= h\in G).
    \end{eqnarray*}
   Here, we applied Corollary  \ref{corequiv} (with $G$ replaced by $C(h)$) to get the penultimate equivalence, and  Lemma \ref{lemmacent} to get the last equivalence.      \end{proof}

\medskip

\section{Some classes of $FSZ$-groups}\label{ssexamples}

We start by applying Theorem \ref{thmcent1} to show that various families of groups $G$ are $FSZ^+$.\

\begin{example} {\rm \ {\bf PSL$_2$(q)}.\ 
Abelian groups are obviously $FSZ$, and dihedral groups are too \cite{K}.\ Because
$PSL_2(q)\ (q\geq 2$ a prime power) has the property that every nonidentity element has an
abelian or dihedral centralizer, we can conclude that $PSL_2(q)$ is an $FSZ^+$-group.\
($PSL_2(q)$ simple if $q\geq 4$.)} 
\end{example}

\begin{example} \label{CA} {\rm  {\bf CA-groups}.\
$G$ is a $CA$-$group$ if the centralizer of every nonidentity
element is \emph{abelian}.\ By Theorem \ref{thmcent1}, every $CA$-group is
$FSZ^+$.\ With more effort, we can give explicit formulas  for, and interpretations of, the indicators in this case.
}\end{example}

First we remark that if $g \neq 1$, then $G_n(u, g)=\{a\in C\ | \ a^n=g \}$, and if $G_n(u, g)\neq \phi$, then $u^n = 1.$ This holds since if $a \in G_n(u, g)$, then $au^{-1}$ and $a$ commute with $g$, hence with each other.\ Thus $a\in C, (au^{-1})^n=a^nu^{-n}=a^n$, and $u^n=1$.  

\medskip \noindent
Case 1.\ $u^n\not= 1$.\ In this case, the remark shows that
$G_n(u, g)= \phi$ if $g\not= 1$.\ Thus, $\zeta_n(g)=0$ for $g\not= 1$, and of course we also have
$\zeta_n(1) = |G_n(u, 1)|$.\ Next, $C=C(u)$ acts by conjugation on $G_n(u, 1)$, and by definition the \emph{permutation character} $\pi$ of $C$ furnished by $G_n(u, 1)$ satisfies
$\pi(g) = |\{x\in G_n(u, 1)\ | \ gxg^{-1}=x\}|$ for $g\in C$. 

We assert that $\zeta_n=\pi.$ 
Certainly $\pi(1)=|G_n(u, 1)|$.\ Furthermore, if $1\not= g\in C$ and $x\in G_n(u, 1)$, then $x$ and $u$ both commute with $g$, hence with each other.\ But then
$(xu^{-1})^n=x^n$ implies that $u^n=1$, contradiction.\ So there is no such $x$, whence
$\pi(g)=0$ for $g\not= 1$.\ Having shown that $\zeta_n(g)=\pi(g)$ for all $g\in C$, our assertion is proved.

\medskip
By a standard (and straightforward) result, a character $\mu$ of
$C$ that vanishes away from the identity element is necessarily a multiple of the
\emph{regular representation} $\rho_C$ of $C$, and the multiplicity is
$\mu(1)/|C|$.\ Since this applies with $\mu=\zeta_n$, we have $\zeta_n = (\zeta(1)/|C|)\rho_C$.\
But the decomposition of the regular representation into irreducibles is standard, namely
$\rho_C = \sum_{\eta\in I(C)}\eta(1)\eta$.\ Putting all of this together, we find that
\begin{eqnarray*}
\zeta_n = \frac{|G_n(u, 1)|}{|C|}\sum_{\eta\in I(C)}\eta(1)\eta.
\end{eqnarray*}
Then comparison with (\ref{zetaform}) shows that
\begin{eqnarray*}
\nu_n(\chi_{\eta})=\frac{|G_n(u, 1)|\eta(1)}{|C|}.
\end{eqnarray*}
So either the indicators $\nu_n(\chi_{\eta})$ are all \emph{positive}, or else $G_n(u, 1)=\phi$ and they 
are all $0$.

\medskip \noindent
Case 2.\ $u^n=1$.\ In this case $C\subseteq G_n(u)$, but once this is taken into account
 the arguments of Case 1 carry over.\ Thus by the previous remark, $\zeta_n(g) = |\{a \in C\ | \ a^n=g\}|$ when 
 $g \not= 1$.\ Comparing this with (\ref{zeta'form}), we see that $\zeta_n$ and $\varphi^C_n$ agree away 
 from the identity, so that $\zeta_n-\varphi^C_n$ is a multiple of $\rho_C$.\ The multiple in question is
\begin{eqnarray*}
(\zeta_n(1)-\varphi^C_n(1))/|C|= (|G_n(u, 1)|-|\{a\in C \ |\ a^n=1\}|)/|C|,
\end{eqnarray*}
the number of
orbits of the conjugation action of $C$ on $G_n(u, 1)\setminus{\{a\in C \ |\ a^n=1\}}$.
As before, this leads to
\begin{eqnarray*}
\zeta_n-\varphi^C_n =\frac{|G_n(u, 1)|-|\{a\in C \ |\ a^n=1\}|}{|C|}\sum_{\eta\in I(C)} \eta(1)\eta, 
\end{eqnarray*}
and
\begin{eqnarray*}
\nu_n(\chi_{\eta}) &=& \nu^C_n(\eta) + \frac{\eta(1)(|G_n(u, 1)|-|\{a\in C \ |\ a^n=1\}|)}{|C|}.
\end{eqnarray*}

\medskip

\begin{example}\label{reg} {\rm {\bf Regular $p$-groups}.\ Fix a prime $p$.\ For an introduction
to \emph{regular} $p$-groups, see \cite{H}, Chapter 12.4.\
All we need to know is that in a regular $p$-group,
the equation $a^k=b^k$ holds if,  and only if,  $(ab^{-1})^k=1$
for $k\in \mathbb{Z}$ \cite[Theorem 12.4.4]{H}.\ We show that a regular $p$-group $G$ is $FSZ^+$.}

\end{example}

\medskip
If $a \in G$ then $(au^{-1})^n=a^n \Leftrightarrow (au^{-1}a^{-1})^n=1 \Leftrightarrow u^n=1$.\
This shows that  either $u^n\not= 1$ and $G_n(u)=\phi$; or
$u^n=1$ and $G_n(u)=G$.\ Now Lemma \ref{power} applies, and shows that
$G$ is $FSZ$.\ But every subgroup of a regular $p$-group is also regular 
(loc.\ cit.), so Theorem \ref{thmcent1} now tells us that $G$ is $FSZ^+$.\ We can again
make the indicators more explicit.\ If $G_n(u)=\phi$ then
$\zeta_n$ vanishes identically and the corresponding indicators
$\nu_n(\chi_{\eta})=0$.\ On the other hand, if $u^n=1$ then
$G_n(u, g)=\{a\in G\ | \ a^n=g\}$.\ Since $\zeta_n(g) = |G_n(u, g)|$, this means that for $g\in C$
the values of $\zeta_n$ \emph{coincide} with those of the generalized character $\varphi^G_n=\sum_{\eta} \nu^G_n(\eta)\eta$ (\ref{zeta'form}), i.e.,
\begin{eqnarray*}
\zeta_n = Res^G_C \varphi^G_n.
\end{eqnarray*}

\medskip
We note that all $p$-groups of order \emph{less than} $p^{p+1}$, nilpotence class \emph{less than} 
$p$, or exponent $p$, are regular \cite{H}, and therefore FSZ.\ We will see later that there 
are groups of order $5^6$ and exponent $5^2$ which are \emph{not} $FSZ$. 

\medskip
\begin{example} \label{wreath} {\rm \ {\bf Wreathed products $Z_p \wr Z_p$}.\
This $p$-group (the regular wreathed product of $Z_p$ with itself), is an irregular
$p$-group of order $p^{p+1}$.\ We will show that it is $FSZ^+$, so that the class of
$FSZ$ $p$-groups is \emph{strictly larger} than the class of regular $p$-groups.}
\end{example}

\medskip
$G$ has a normal elementary abelian $p$-subgroup
$V$ of order $p^p$ (which we consider as a $GF(p)$-vector space of rank $p$) and a cyclic subgroup $J=\langle j \rangle$ of order $p$ such that $j$ cyclically
permutes a basis of $V$. We identify $j$ with the $p\times p$ Jordan block
\begin{eqnarray*}
j = \left(\begin{array}{cccc} 0&  &  & 1 \\1 &  &  &  \\ & \ddots &  &  \\ &  & 1 &0 \end{array}\right)
\end{eqnarray*}
Then $G = J \ltimes V$, and multiplication in $G$ is $(j^r, x)(j^s, y)=(j^{r+s}, j^sx+y)$.\
The subgroup of $G$ generated by $p^{th}$ powers coincides with the center
$Z:=Z(G)$. $Z\subseteq V$ is $1$-dimensional and spanned by the all $1$'s vector $\mathbf{1}$.
The operator $k:= 1+j+\hdots+j^{p-1}$ plays a r\^{o}le below;
$k$ is the all $1$s matrix and $kV=Z$.
 
\medskip
Because $G$ has exponent $p^2$, in proving that $G$ is $FSZ$ it suffices to consider 
only the cases $n=p, p^2$, and of these only $n=p$ is in doubt. Fix some $u^{-1}:=(j^d, t)\in G$.
By what we have said, $G_p(u, g)=\phi$ unless $g\in Z$.\ We show that 
(\ref{gpeqns}) holds with $n=p$ and $1\not= g\in Z$.

\medskip
An element $a=(j^i, v)$ satisfies $a^p = (1, kv)$ if $j^i\not= 1$, and otherwise $a^p=1$.
Moreover, $au^{-1}=(j^{i+d},j^dv+t )$. Therefore, $(au^{-1})^p=a^p=g$ if, and only if,
the four conditions $i\not \equiv 0(\mbox{mod}\ p), i+d\not \equiv 0(\mbox{mod}\ p),  kv=g, k(j^dv+t)=g$
all hold.\ But $k(j^dv+t)=kv+kt$ (because $k$ is the all $1$s matrix and $j^d$ a permutation matrix).\ So if the four conditions hold then we must
have $kt=0$. Turned around, if $u^{-1}=(j^d, t)$ and $kt\not= 0$ then
$G_p(u, g)=\phi$ for any nonzero $g \in Z$.\ If, on the other hand, $kt=0$, then
$G_p(u, g)= \{(j^i, v)\ | \ i \not \equiv 0, -d (\mbox{mod}\ p), kv=g \}$ 
has cardinality $(p-1)p^{p-1}$ or $(p-2)p^{p-1}$ depending on whether or not
$d \equiv 0(\mbox{mod}\ p)$.\ So in all cases, the desired property of
$G_p(u, g)$ is established.\ Note that the hypotheses of Lemma \ref{power} 
do \emph{not} hold in this case, so they are sufficient but not necessary.

\medskip \noindent
\begin{example}\label{DP} {\rm \ {\bf Direct products}.\ It follows easily from the definition 
that, in an obvious sense,  $|G_n(u, g)|$ is \emph{multiplicative} over direct products of groups.\
Thus the class of $FSZ$ groups is closed under direct products; for example, 
a nilpotent group, all of whose Sylow subgroups are regular, is $FSZ$.\ On the other hand,
the direct product of two groups, one of which is $FSZ$ and one which is not,
is itself \emph{not} $FSZ$.\ So once we have found a non$FSZ$ group $F$, say,
there are large numbers of non$FSZ$ groups having $F$ as a direct factor.}
\end{example}

\medskip \noindent
\begin{example}\label{regnilp} {\rm\ {\bf Regular CN-groups}.\ $G$ is a $CN$-group if the centralizer of every nonidentity
element is nilpotent (i.e., a direct product of $p$-groups).\ We call $G$ a \emph{regular}
CN-group if the centralizer of every nonidentity element is a direct product of regular $p$-groups.\ 
By Theorem \ref{thmcent1} together with Examples \ref{reg} and \ref{DP} above, every
regular CN-group is FSZ.}
\end{example}

\medskip \noindent
\begin{example} {\rm {\bf Mathieu group $M_{11}$}.\ We show that $M_{11}$ is $FSZ$ by applying
Lemma \ref{lemmacent}.\ Indeed, it is readily checked \cite{AT} that centralizers of elements of
order $5$ or greater than $6$ are $FSZ$, and the result follows.} 
\end{example}

\medskip

\section{A sufficiency condition and groups of special orders}

We start by showing that  (\ref{Gniso}) holds automatically for \emph{some} values of $m$.\
\begin{lemma}\label{L1}
Suppose that $m \in \ZZ$ satisfies
$gcd(m, |G|)=1$ and
$m\equiv \pm 1\ (\mbox{mod}\ n$). Then for all pairs of commuting elements $u, g\in G$, there is an isomorphism
of $C(u, g)$-sets
$$
G_{n}(u,g) \rightarrow G_n(u, g^{m}), \ \ a\mapsto a^m.
$$
\end{lemma}
\begin{proof}
Let $a \in G_n(u, g)$, so that $(au^{-1})^n=a^n=g$, and suppose first that
$m\equiv 1\ (\mbox{mod}\ n)$.\ Then
\begin{eqnarray*}
(a^{m}u^{-1})^n & = & (a^{m-1}au^{-1})^n \\
& = & a^{(m-1)n}(au^{-1})^n  \ \ \  {\rm \ (because \ }a^{m-1}\in \langle a^n \rangle \subseteq C(u))\\
& = & a^{mn}.
\end{eqnarray*}
This argument shows that  the $m^{th}$ power map
$\varphi_m: a\mapsto a^m$ induces an \emph{injection} $G_n(u, g)\rightarrow G_n(u, g^m)$.\ If $m'$ is the inverse of $m\ (\mbox{mod}\ n)$, then $\varphi_{m'}$ is the inverse of $\varphi_m$,
so that $\varphi_m$ is a \emph{bijection}.\ Moreover, because conjugation commutes with power maps,
$\varphi_m$ is also a morphism of $C(u, g)$-sets.

\medskip
 This completes the proof of the Lemma when $m\equiv 1\ (\mbox{mod}\ n)$.\ If $m\equiv -1\ (\mbox{mod}\ n)$, the result follows
from the case $m\equiv 1\ (\mbox{mod}\ n)$ together with 
Lemma \ref{z's}, because the bijection of Lemma \ref{z's} is $C(u, g)$-equivariant.
\end{proof}

\medskip
For some groups, Lemma \ref{L1} tells us all we need to know in order to deduce the $FSZ^+_n$ property.\ This is the essential content of
\begin{theorem}\label{thmqnt'}
Suppose that for each element $a\in G$ we have 
$$ o(a^n)/o(a^{n^2})\in \{1,2,3,4,6\}.$$ Then
$G$ is an $FSZ^+_n$-group.
\end{theorem}
\begin{proof}\ It suffices to show that (\ref{Gniso}) holds for all commuting pairs of elements $u, g$
and integers $m$ coprime to $|G|$.\ For in this case,
 the Remark that follows (\ref{Gnpermiso}) shows that $G$ is $FSZ^+_n$.

\medskip
Suppose that $o(g)=t$.\ If $G_n(u, g^m)=\phi$ for all $m$ there is nothing to prove,
so assume that $a\in G_n(u, g^m)$ for some $m$.\ Thus $a^n = g^m$ has order $t$, and 
$a^{n^2}=g^{mn}$ has order $t/gcd(t, n)$.\ Thus
$o(a^n)/o(a^{n^2}) = gcd(t, n)$.\ As a result, the hypothesis of the Theorem tells us that
\begin{eqnarray}\label{gcdt}
gcd(t, n) \in \{1, 2, 3, 4, 6\}.
\end{eqnarray}

Set $d=gcd(t, n)$.\ Because $m$ is coprime to $t$ (Cauchy's theorem)
it is also coprime to $d$.\ Because of (\ref{gcdt}), it follows that
\begin{eqnarray*}
m \equiv \pm 1\ (\mbox{mod}\ d).
\end{eqnarray*}

There are integers $a, b$ such that $at+bn=d$.\ Then for some integer $k$ we have
$m = \pm1 +kd=\pm 1+kat+kbn$ and therefore $g^m = g^{\pm1 +kbn}$.\ Now by
Lemma \ref{L1} we see that there is an isomorphism of $C(u, g)$-sets  
$G_n(u, g) \rightarrow G_n(u, g^m)$, as required.\
This completes the proof of the Theorem.
\end{proof}

\medskip

The Theorem has some interesting consequences.
\begin{corollary}\label{special}
Suppose that the order of each element of $G$ divides $2^\alpha 3^\beta q$, where 
$q$ is a squarefree integer coprime to $6$ and either $\alpha,\beta\leq 3$ or $\beta\leq 1$ and $\alpha\leq 5$.\ Then $G$ is $FSZ^+$.
\end{corollary}
\begin{proof}
The  hypothesis of the previous theorem hold for all $n$.
\end{proof}

\begin{corollary}\label{corsmalln} Suppose that $n\in\{1, 2, 3, 4, 6\}$.\ Then 
every group $G$ is $FSZ^+_n$. 
\end{corollary}
\begin{proof}\ We always have 
$o(a^n)/o(a^{n^2})|n$ for $a\in G$ (cf.\ the proof of Theorem \ref{thmqnt'}).
\end{proof}

\begin{corollary}
All 
groups of order less that $100$ are $FSZ^+$. $\hfill \Box$
\end{corollary}

Apropos of our introductory comments that some of our stated results are not best-possible, this is certainly true of the preceding Corollary.\ Because all groups of order $512=2^9$ are FSZ (cf.\ Section 7), one could  check all groups up to order $728=3^6-1$ rather easily.\ It is very likely that the  nonFSZ-groups of \emph{least order} are irregular $p$-groups of order$2^k (10\leq k \leq 13), 3^k (6\leq k\leq 8)$, or $5^6$.

\medskip \noindent
\begin{corollary} {\rm {\bf Mathieu groups $M_{11}$, $M_{12}$}.\ We have seen before that the Mathieu Group $M_{11}$ is FSZ. We note that we could use the fact that the exponent of both $M_{11}$ and $M_{12}$ is $2^3\cdot 3\cdot 5\cdot 11$ and apply Corollary \ref{special} to see that, in fact, both $M_{11}$ and $M_{12}$ are $FSZ^+$.}
\end{corollary}


\section{Symmetric  groups}

\begin{theorem}\label{thmsymm} The symmetric group $S_N$ is an $FSZ$-group.
\end{theorem}
\begin{proof}\ We have to show that (\ref{gpeqns}) holds for all commuting pairs
$u, g\in S_N$ and all $m$ coprime to $N!$.\ There is nothing to prove if $g=1$, so assume
that $1\not= g$ is a product of $e_i$ cycles of length $i$.\ Thus $\sum_i ie_i=N$,
and $C(g)$ is contained in a subgroup of $S_N$ isomorphic to
$S_{e_1}\times S_{2e_2}\times \hdots \times S_{Ne_N}$.\ Using induction on $N$
and the multiplicativity of $G_n(u, g)$ over direct products (cf.\  Subsection 4, Example 5),
we may, and shall, assume that $g$ is a product of $e$ cycles 
of length $r$, with $N=er$.\ Then $C(g)=S_e\wr Z_r = S \ltimes V$ where $S=S_e$ and
$V=Z_r^e$.\ We use additive notation for $V$, and write  $g=(\rm{1}, g_0,  \hdots, g_0)$ with
$(g_0,\dots,g_0)\in Z_r^e$ for some 
\emph{generator} $g_0$ of $Z_r$.\ In effect, $V$ is a right
$Z_rS$-module with $S$-action denoted exponentially by $v^s$ (which is right conjugation
$s^{-1}vs$ in $C(g)$), and multiplication in $C(g)$ is
\begin{eqnarray*}
(s, v)(t, w) = (st, v^t+w).
\end{eqnarray*}
Then
\begin{eqnarray*}
(s, v)^n  = (s^n, v+v^s+\hdots+v^{s^{n-1}}).
\end{eqnarray*}

\medskip
Let $(s, v)=sv \in C(g)$ ($s\in S, v\in V$), and let $s=s_1\hdots s_q$ be a representation of
$s$ as a product of disjoint cycles in $S_e$.\ There is a corresponding decomposition
of $v = v_1+\hdots+ v_q$ such that for each index $i$, $s_i$ has length $f_i$,
$v_i\in W_i\cong Z_r^{f_i}$, and $sv=(s_1v_1)\hdots(s_qv_q)\in (S_{f_1}\wr Z_r)\times \hdots \times (S_{f_q}\wr Z_r)$.

\bigskip
Now suppose that $(s, v)^n = g^m = (1, mg)=(1, mg_0, \hdots, mg_0)$.\ Then for each index $i$
we have $(s_i, v_i)^n=(1, mg_0, \hdots mg_0)\in S_{f_i}\wr Z_r$.\
Then $s_i^n=1$, whence $f_i|n$; moreover $(mg_0, \hdots, mg_0)=v_i+v_i^{s_i}+\hdots+v_i^{s_i^{n-1}}
=(n/f_i)(v_i+v_i^{s_i}+\hdots+v_i^{s_i^{f_i-1}})$.\ We can write this more compactly as
\begin{eqnarray}\label{norm1}\num
m(g_0, \hdots, g_0)
=(n/f_i)N_{s_i}(v_i) \in Z_r^{f_i},
\end{eqnarray}
where $N_{s_i}$ denotes \emph{norm} with respect to $\langle s_i \rangle$.\ Note that
because the left hand side of (\ref{norm1}) has order exactly $r$ (as an element of $V$), then  $gcd(n/f_i, r)= 1$.\
Setting $l=gcd(n, r)$, it follows that
 the condition $l|f_i$ (for all indices $i$) is \emph{necessary} in order that there
be \emph{any} solutions of the equation $(s, v)^n=g^m$.

\medskip
It is also sufficient: if $gcd(n/f_i, r)=1$, we can always  find $v'_i\in W_i:=Z_r^{f_i}$ whose norm
is $(n/f_i)^{-1}m(g_0, \hdots, g_0)$.\ The element
$v_i''=(g_0, 0, \hdots, 0, g_0, 0, \hdots, 0 , \hdots, g_0, \hdots, 0)$, of length $f_i$ and with 
$g_0$'s occuring after $l-1$ zeros, has norm $(f_i/l)(g_0, \hdots, g_0)$, and we take
$v'_i =  m(n/l)^{-1}(v_i'')^{s}.$\ Here inverses are taken modulo $r$.

\medskip
The set of \emph{all} $v_i \in W_i$ satisfying (\ref{norm1}) is
\begin{eqnarray*}
v_i'+ker N_{s_i} = v_i'+[W_i, s_i],
\end{eqnarray*}
where $[W_i, s_i]\subseteq W_i$ is spanned by $w-w^{s_i}, w\in W_i$.\ Because $s_j$ acts
trivially on $W_i$ for $i\not= j$, we have $[W_i, s_i] = [W_i, s]$.\ It follows that, for fixed
$s\in S$ as above, the set of all $v\in V$ satisfying $(s, v)^n=g^m$ is
\begin{eqnarray*}
v_1'+\hdots+v_q'+[V, s].
\end{eqnarray*}

\medskip
Fix a second element $(t, w)\in C(g)$, so that $(s, v)(t, w)=(st, v^t+w)$.\ By the foregoing,
the set of $x\in V$ such that $(st, x)^n=g^m$ is the coset
\begin{eqnarray*}
x_1'+\hdots+x_p'+[V, st],
\end{eqnarray*}
where $st$ is a product of $p$ disjoint cycles of lengths $k_1, \hdots, k_p$
and $x_i'= m(n/l)^{-1}(x_i'')^{st}$ with
$x_i''=(g_0, 0, \hdots, 0, g_0, 0, \hdots, 0 , g_0, \hdots0)$ similarly to the previous case.\ Hence, for fixed $s\in S$ and fixed $(t, w)\in C(g)$, the  set of elements
$v\in V$ such that $(s, v)^n = ((s, v)(t, w))^n=g^m$ is
\begin{eqnarray*}
&&(x_1'+\hdots+x_p'+[V, st])\cap (w+(v_1'+\hdots+v_q'+[V, s])^t) \\
=&&m(n/l)^{-1}\left\{\left(\sum_{i=1}^p(x_i'')^{st}+[V, st]\right)\cap
 \left(m^{-1}(n/l)w^{t^{-1}}+\sum_{i=1}^q(v_i'')^s+[V, s]\right)^t\right\}.
\end{eqnarray*}

\medskip
Now we have gone out of our way to arrange that  $\sum_ix_i'' = \sum_i v_i''$.\ Because
$s$ and $st$ leave invariant $[V, s]$ and $[V, st]$ respectively, it follows that the cardinality of the last set is 
\begin{eqnarray*}
&&\left|\left(\sum_{i=1}^px_i''+[V, st]\right)^{st}\cap
 \left(m^{-1}(n/l)w^{(st)^{-1}}+\sum_{i=1}^qv_i''+[V, s]\right)^{st} \right| \\
 =&& \left|\left([V, st]\right)\cap
 \left(m^{-1}(n/l)w^{(st)^{-1}}+[V, s]\right) \right|.
\end{eqnarray*}
This equals $0$
if $m^{-1}(n/l)w^{(st)^{-1}}\notin [V, st]+[V, s]$, and equals $|[V, st]\cap[V, s]|$ if \\
$m^{-1}(n/l)w^{(st)^{-1}}\in  [V, st]+[V, s]$, 
and in both cases
the condition is independent of $m$ coprime to $r$.\ This completes the proof of the
Theorem.
\end{proof}

\medskip
The proof of the Theorem also yields
\begin{corollary} The alternating group $A_N$ is an $FSZ$-group.
\end{corollary}
\begin{proof}\ Keep the same notation as before, now with $g \in A_N$.\
Then $C_{A_N}(g)=C_{S_N}(g) \cong S_e\wr Z_r$ if $r$ is even, and $C_{A_N}(g)\cong A_e\wr Z_r$ if $r$ is odd.\ In either case the proof of the Theorem carries over \emph{verbatim}, because 
once $s, t, w$ are fixed the rest of the calculation takes place in $V\subseteq A_N$. 
\end{proof}

\medskip
Pavel Etingof has communicated to us \cite{E} the proof of a natural extension of Theorem
\ref{thmsymm}, namely 
\begin{theorem}\label{thmEsymm} (Etingof) 
$S_N\wr A$ is an $FSZ$-group for any finite abelian group $A$. 
\end{theorem}

\medskip
The proof, which we omit, runs along similar lines to that of Theorem \ref{thmsymm},
but is more elaborate.\ As an application, recall that every centralizer in
$S_N$ is a direct product of wreathed products of the type occuring in
Theorem \ref{thmEsymm}.\ Then by Theorem \ref{thmcent1} we deduce
\begin{corollary} $S_N$ is an $FSZ^+$-group.
\end{corollary}

\section{Counterexamples}

In this section we consider groups of order $5^6=15625$. We present two GAP programs, one optimized for testing these $5$-groups, and the second which is optimized to be used to test the FSZ property any arbitrary group $G$. 

\medskip
Using Corollary \ref{corequiv}, the first \cite{GAP} program can be used to test integrality of indicators for doubles of $p$-groups of exponent $p^2$. On a dual core 3.33 Mz computer with 8Gb of Ram, after approximately 6 hours of running for the 555 groups of exponent 25 among the 684 groups of order $15625=5^6$, it found 32 groups whose indicators for the Drinfel'd double are not all integers (only indicator $\nu_5$ can be non-integer, and all the other groups of order $5^6$ can easily be seen to be FSZ). These are the groups SmallGroup(12625,$k$) (in the small library) for $k$ being the ID numbers 632--635, 637, 638, 640--650, 653--655, and 657--668. We note that groups 662, 665 and 668 have few conjugacy classes, many of which give rise to representations of the Drinfeld double with non-integer indicators, and also several groups $G$ among these have a large number of conjugacy classes and still produce representations of the double with non-integer indicators. 

\medskip
One particular such group is ${\rm SmallGroup}(15625,668)$. The GAP program takes parameters $G,p,m$ and is designed to test the integrality of the $m$-indicators of representations of $D(G)$ induced from cyclic subgroups of prime order $p$ (only). This is enough to test the FSZ property for a p-group of exponent $p^2.$ The command $${\rm FSInd}({\rm SmallGroup}(15625,668),5,5)$$ tests the integrality of indicators $\nu_5$ for representations induced from cyclic groups of order $5$ and determines those conjugacy classes (that is, their GAP ID) for which there are associated representations of the double which have non-integral indicators. It runs in a few seconds on the computer hardware configuration mentioned above for group ${\rm SmallGroup}(15625,668)$. 

\medskip
By Lemma \ref{gal} all indicators are in $\QQ(\zeta_5)$, and since they are also real by Remark \ref{p.real}, the non-integer indicators are algebraic 
integers in $\QQ(\sqrt 5)\setminus \QQ$. We confirmed this using a GAP program of Rebecca Courter \cite{C}, which computes values for the FS-indicators of irreducible representations of the double of ${\rm SmallGroup}(15625,668)$.

-----------------------------------------------------

\begin{verbatim}
FSInd:=function(G,m,p)
local GG,list,Gm,slist,gucoef,u,Cl,T,H,nu; 

GG:=EnumeratorSorted(G);; 
elG:=Size(G); 
Gm:=[];;
list:=[];; slist:=[]; gucoef:=[];
for i in [1..elG] do Gm[i]:=Position(GG,GG[i]^m); od; 

for i in [1..elG] do list[i]:=[]; slist[i]:=0; od;
for i in [1..elG] do Add(list[Gm[i]],i); od;

for Cl in ConjugacyClasses(G) do
    u:=Representative(Cl);
             Print("u=",Position(GG,u),"; ");
    for i in [1..elG] do gucoef[i]:=0; od;   
    for i in [1..elG] do 
        for j in list[i] do if Position(GG,GG[j]*u) in list[i] 
        then gucoef[i]:=gucoef[i]+1; fi; od;

    od;    

    H:=1;
    Print("INDICATORS:\n");
    for i in [1..elG] do T:=1;; if Order(GG[i])=p then
          nu:=0;
          for j in [0..Order(GG[i])-1] do 
          nu:=nu+ (E(p)^j)*gucoef[Position(GG,GG[i]^j)]; od;
          Print(nu,", "); 
          for j in [2..Order(GG[i])-1] do if
          gucoef[i]=gucoef[Position(GG,GG[i]^j)] then; 
                    else T:=0;; H:=0;; fi; od; fi;     
          #This checks it for rep's induced 
          from prime order subgroups;
          
          od;
          if H=1 then Print("\n OK - Integer"); 
             else Print("\n !!! N O N  -  I N T E G E R !!!"); fi; 
          Print("\n");

od;

return;
end;

\end{verbatim}

-------------------------------------------------------------------------------

Our second program provides a general computational test for whether the indicators $\nu_m$ of {\it all} the representations of the double of a group $G$ are integers; it takes the integer $m$ and the group $G$ as parameters. It uses the test provided by Corollary \ref{corequiv}, and it is optimized to attempt the least possible number of computations of products in $G$ (but uses more memory). This is necessary since products of elements (usually represented as compositions of permutations) tend to take the most time under GAP.

The algorithm runs as follows. The program first creates a list [Gm] with all the $m$'th powers of elements of $G$ (GAP works with list position), and then another variable [list] such that which for each element $x\in G$, list[x]  retains all $h$ for which $h^m=x$ (this saves doing computations of the type $\underbrace{h*h*h\dots *h}$ each time we compute an indicator). Then, we create the list variable [glist], which will retain the sets $\{g^{k}|\gcd({\rm ord}(g),k)=1\}$ - elements generating the same cyclic subgroup, on which we have to test the condition of Corollary \ref{corequiv}. Next, for each conjugacy class Cl of $G$, we pick an element of Cl and we create a list [gucoef] which for each element $g$ of $G$ will count the number of solutions of the equation $h^m=(hu)^m=g$ (obviously, for the purpose of testing integrality of all indicators, we can just work with $u$ instead of $u^{-1}$). This uses the previously created variable [list] for the 
$m^{th}$roots of $g$ (so that we don't need to perform supplementary multiplications). Once this is done, using 
Corollary \ref{corequiv}, one only needs to check that the number of these solutions is constant on the sets contained in the variable [list]. The printout also shows the GAP ID of some representative $u$ of each conjugacy class and whether $Cl(u^{-1})$ produces some representation of $D(G)$ with non-integer indicator $\nu_m$. The program could be slightly optimized to jump out of the loop whenever a non-integer indicator is found, and it can be easily modified to find the indicators for all the representations.

\begin{verbatim}
FSIndG:=function(G,m)
local GG,list,Gm,gucoef,u,Cl,H,expo, 
      glist,aux,listaux,Gaux,g,h,og,gaux; 

GG:=EnumeratorSorted(G);; expo:=Exponent(G);
elG:=Size(G); 
Gm:=[];;
list:=[];; gucoef:=[];
for i in [1..elG] do Gm[i]:=Position(GG,GG[i]^m); od; 
         #create the list of m-powers of elements
for i in [1..elG] do list[i]:=[]; od;
for i in [1..elG] do Add(list[Gm[i]],i); od;
         #create for each element x the list of its m-"radicals" h
glist:=[]; og:=Size(G); Gaux:=G;
while og>0 do g:=Representative(Gaux); aux:=[];
           if Order(g) * m < expo+1 then 
           for j in PrimeResidues(Order(g)) do 
           h:=g^j; Add(aux,Position(GG,h)); og:=og-1; 
           Gaux:=Difference(Gaux,[h]); od; 
           Add(glist,aux); 
           else og:=og-1; Gaux:=Difference(Gaux,[g]); fi;  
           od; 
         #this creates the list "glist" whose members are sets 
         #of elements of G generating the same subgroup
for Cl in ConjugacyClasses(G) do
    u:=Representative(Cl);
             Print("u=",Position(GG,u),"; ");
    for i in [1..elG] do gucoef[i]:=0; od;   
    for i in [1..elG] do 
        for j in list[i] do if Position(GG,GG[j]*u) in list[i] 
        then gucoef[i]:=gucoef[i]+1; fi; od;

    od;    
         #gucoef[g] counts the number of solutions of h^m=(hu)^m=g

    H:=1;    
    for i in glist do 
             #nu:=0;
             for j in i do if gucoef[i[1]]=gucoef[j] then; 
                       else H:=0;; fi; od;     

             #This uses our results to check integrality by testing  
             #that the number of solutions in H_m(g) of h^m=g^k and 
             #h^m=g^l is the same for gcd(k,ord(g))=gcd(l,ord(g))
             #H retains whether the integer condition is verified
             
             od;
             if H=1 then Print(" OK - Integers"); 
                else Print(" !!! N O N  -  I N T E G E R !!!"); fi; 
             Print("\n");
od;

return;
end;
\end{verbatim}

Using this last program, we observed that all the 2-groups up to order 512 are FSZ. Since 2-groups can in general behave differently than other p-groups, it is natural to ask whether perhaps all 2-groups are FSZ.

\medskip
It is also known that all the indicators of the representations of the group $S_n$ are {\it non-negative} integers 
\cite{Scf}. Thus one might ask whether this is true for $D(S_n)$ as well. This has been shown in \cite{C} for 
$n \leq 10$. It is also known that for any dihedral group, all indicators of its double are non-negative \cite{K}.

\end{document}